\documentclass[12pt,reqno]{amsart}

\usepackage{amssymb,amsmath,amsthm,graphicx,xcolor,mathtools,mathrsfs,tabularx,bbm,tikz,url}
\usepackage{mathabx}\changenotsign
\usepackage{dsfont}
\usepackage[shortlabels]{enumitem}
\usepackage{lmodern}
\usepackage[babel]{microtype}
\usepackage[british]{babel}
\usepackage[utf8]{inputenc}
\allowdisplaybreaks

\usepackage[backref, hypertexnames=false]{hyperref} 
\hypersetup{
	colorlinks,
	linkcolor={red!60!black},
	citecolor={green!60!black},
	urlcolor={blue!60!black}
}

\def\ssign{\textsection\nobreak\hspace{1pt plus 0.3pt}}
\makeatletter
\let\origsection=\section 
\def\mysection{\@mystartsection{section}{1}\z@{.7\linespacing\@plus\linespacing}{.5\linespacing}{\normalfont\scshape\centering\ssign}}
\def\section{\@ifstar{\origsection*}{\mysection}}
\def\appendix{\par\c@section\z@ \c@subsection\z@
	\let\sectionname\appendixname
	\let\section=\origsection
	\def\thesection{\@Alph\c@section}}
\def\@mystartsection#1#2#3#4#5#6{\if@noskipsec \leavevmode \fi
	\par \@tempskipa #4\relax
	\@afterindenttrue
	\ifdim \@tempskipa <\z@ \@tempskipa -\@tempskipa \@afterindentfalse\fi
	\if@nobreak \everypar{}\else
	\addpenalty\@secpenalty\addvspace\@tempskipa\fi
	\@dblarg{\@mysect{#1}{#2}{#3}{#4}{#5}{#6}}}
\def\@mysect#1#2#3#4#5#6[#7]#8{\edef\@toclevel{\ifnum#2=\@m 0\else\number#2\fi}\ifnum #2>\c@secnumdepth \let\@secnumber\@empty
	\else \@xp\let\@xp\@secnumber\csname the#1\endcsname\fi
	\@tempskipa #5\relax
	\ifnum #2>\c@secnumdepth
	\let\@svsec\@empty
	\else
	\refstepcounter{#1}\edef\@secnumpunct{\ifdim\@tempskipa>\z@ \@ifnotempty{#8}{\@nx\enspace}\else
		\@ifempty{#8}{.}{\@nx\enspace}\fi
	}\@ifempty{#8}{\ifnum #2=\tw@ \def\@secnumfont{\bfseries}\fi}{}\protected@edef\@svsec{\ifnum#2<\@m
		\@ifundefined{#1name}{}{\ignorespaces\csname #1name\endcsname\space
		}\fi
		\@seccntformat{#1}}\fi
	\ifdim \@tempskipa>\z@ \begingroup #6\relax
	\@hangfrom{\hskip #3\relax\@svsec}{\interlinepenalty\@M #8\par}\endgroup
	\ifnum#2>\@m \else \@tocwrite{#1}{#8}\fi
	\else
	\def\@svsechd{#6\hskip #3\@svsec
		\@ifnotempty{#8}{\ignorespaces#8\unskip
			\@addpunct.}\ifnum#2>\@m \else \@tocwrite{#1}{#8}\fi
	}\fi
	\global\@nobreaktrue
	\@xsect{#5}}
\makeatother
\usepackage{setspace}

\usepackage{geometry}
\let\setminus=\smallsetminus
\let\emptyset=\varnothing

\makeatletter
\def\moverlay{\mathpalette\mov@rlay}
\def\mov@rlay#1#2{\leavevmode\vtop{   \baselineskip\z@skip \lineskiplimit-\maxdimen
		\ialign{\hfil$\m@th#1##$\hfil\cr#2\crcr}}}
\newcommand{\charfusion}[3][\mathord]{
	#1{\ifx#1\mathop\vphantom{#2}\fi
		\mathpalette\mov@rlay{#2\cr#3}
	}
	\ifx#1\mathop\expandafter\displaylimits\fi}
\makeatother

\usepackage{mathfixs}

\usepackage{hyperref}
\hypersetup{colorlinks, linkcolor={red!50!black}, citecolor={green!50!black}, urlcolor={blue!50!black}}

\usepackage{caption}
\usepackage{subcaption}

\usepackage[mathlines]{lineno}
\usepackage{etoolbox} 

\newcommand*\linenomathpatch[1]{%
	\expandafter\pretocmd\csname #1\endcsname {\linenomath}{}{}%
	\expandafter\pretocmd\csname #1*\endcsname{\linenomath}{}{}%
	\expandafter\apptocmd\csname end#1\endcsname {\endlinenomath}{}{}%
	\expandafter\apptocmd\csname end#1*\endcsname{\endlinenomath}{}{}%
}
\newcommand*\linenomathpatchAMS[1]{%
	\expandafter\pretocmd\csname #1\endcsname {\linenomathAMS}{}{}%
	\expandafter\pretocmd\csname #1*\endcsname{\linenomathAMS}{}{}%
	\expandafter\apptocmd\csname end#1\endcsname {\endlinenomath}{}{}%
	\expandafter\apptocmd\csname end#1*\endcsname{\endlinenomath}{}{}%
}

\expandafter\ifx\linenomath\linenomathWithnumbers
\let\linenomathAMS\linenomathWithnumbers
\patchcmd\linenomathAMS{\advance\postdisplaypenalty\linenopenalty}{}{}{}
\else
\let\linenomathAMS\linenomathNonumbers
\fi

\linenomathpatchAMS{gather}
\linenomathpatchAMS{multline}
\linenomathpatchAMS{align}
\linenomathpatchAMS{alignat}
\linenomathpatchAMS{flalign}
\linenomathpatch{equation}

\usepackage{thmtools}
\usepackage{cleveref}

\theoremstyle{plain}
\newtheorem{theorem}{Theorem}[section]
\crefname{theorem}{Theorem}{Theorems}

\crefname{proposition}{Proposition}{Propositions}

\crefname{corollary}{Corollary}{Corollaries}

\newtheorem{lemma}[theorem]{Lemma}
\crefname{lemma}{Lemma}{Lemmata}

\newtheorem{conjecture}[theorem]{Conjecture}
\crefname{conjecture}{Conjecture}{Conjectures}

\crefname{problem}{Problem}{Problem}

\newtheorem{claim}[theorem]{Claim}
\crefname{claim}{Claim}{Claims}

\crefname{observation}{Observation}{Observations}

\crefname{setup}{Setup}{Setups}

\crefname{fact}{Fact}{Facts}

\crefname{algorithm}{Algorithm}{Algorithms}

\crefname{remark}{Remark}{Remarks}

\crefname{example}{Example}{Examples}

\theoremstyle{definition}

\crefname{definition}{Definition}{Definitions}

\crefname{construction}{Construction}{Constructions}

\crefname{question}{Question}{Questions}

\numberwithin{equation}{section}

\crefformat{section}{\S#2#1#3}
\crefformat{subsection}{\S#2#1#3}
\crefformat{subsubsection}{\S#2#1#3}
\crefrangeformat{section}{\S\S#3#1#4 to~#5#2#6}
\crefmultiformat{section}{\S\S#2#1#3}{ and~#2#1#3}{, #2#1#3}{ and~#2#1#3}

\crefname{appendix}{Appendix}{Appendix}

\crefname{figure}{Figure}{Figures}

\newcommand{\rf}[1]{\cref{#1} (\nameref*{#1})}

\theoremstyle{definition}

\newtheorem*{space-prop}{Space}

\newenvironment{proofclaim}[1][Proof of the claim]{\begin{proof}[#1]}{\end{proof}}

\def\COMMENT#1{}

\usepackage{comment}

\let\polishlcross=\l
\def\l{\ifmmode\ell\else\polishlcross\fi}
 
\newcommand{\es}{\emptyset}
\newcommand{\eps}{\varepsilon}
\renewcommand{\rho}{\varrho}
\newcommand{\sm}{\setminus}
\renewcommand{\subset}{\subseteq}

\newcommand{\Exp}{\mathbb{E}}

\newcommand{\rdeg}{\overline{\deg}}

\newcommand{\cF}{\mathcal{F}}

\newcommand{\cK}{\mathcal{K}}

\newcommand{\cM}{\mathcal{M}}

\newcommand{\cQ}{\mathcal{Q}}

\newcommand{\cS}{\mathcal{S}}

\newcommand{\cU}{\mathcal{U}}
\newcommand{\cV}{\mathcal{V}}
\newcommand{\cW}{\mathcal{W}}

\DeclareMathOperator{\poly}{poly}

\title{Blowing up Dirac's theorem}

\author[R.~Lang]{Richard Lang}

\address[R.~Lang]{
	Departament de Matemàtiques,
	Universitat Politècnica de Catalunya,
	Barcelona, Spain
}
\email{richard.lang@upc.edu}

\author[N.~Sanhueza-Matamala]{Nicolás Sanhueza-Matamala}

\address[N.~Sanhueza-Matamala]{ 
	Departamento de Ingeniería Matemática,
	Facultad de Ciencias Físicas y Matemáticas,
	Universidad de Concepción,
	Concepción, Chile
}
\email{nsanhuezam@udec.cl}

\begin{document}

\begin{abstract}
	We show that every graph $G$ on $n$ vertices with $\delta(G) \geq (1/2+\eps)n$ is spanned by a complete blow-up of a cycle with clusters of nearly uniform size $\Omega(\log n)$.
	The proof is based on a recently introduced approach for finding vertex-spanning substructures via blow-up covers.
\end{abstract}


\maketitle


	\vspace{-0.5cm}

\section{Introduction}

A fundamental problem in extremal combinatorics is to determine the least number of edges that guarantees a copy of a graph $F$ in an $n$-vertex host graph~$G$.
Erdős and Stone~\cite{ES46} proved that host graphs whose density slightly exceeds the \emph{Turán density} $\tau(F)$, meaning $e(G) \geq (\tau(F) +\eps) \binom{n}{2}$, even contain a (complete) \emph{blow-up} of $F$.
Formally, such a blow-up is defined by replacing the vertices of $F$ with vertex-disjoint \emph{clusters} and each edge with a complete bipartite graph whose parts are the corresponding clusters.
The arguments of Erdős and Stone show that we can find blow-ups of $F$ with clusters of size $\Omega(\log n)$, which is optimal as witnessed by dense random graphs.

It is natural, then, to ask if this `blow-up' phenomenon extends to  types of structures other 	than copies of a fixed graph $F$.
We study this question for vertex-spanning subgraphs.
In this setting, one typically replaces conditions on the number of edges with conditions on the minimum degree.
For instance, the classic theorem of Dirac~\cite{Dir52} states that every graph $G$ on $n \geq 3$ vertices with $\delta(G) \geq n/2$ contains a Hamilton cycle.
A consequence of the `Bandwidth Theorem' of Böttcher, Schacht and Taraz~\cite{BST09} is that a minimum degree of $\delta(G) \geq (1/2 + \eps)n$ guarantees a vertex-spanning blow-up of a cycle with clusters of order $\Omega(1)$.
We extend this with the following blown-up version of Dirac's theorem.

\begin{theorem}\label{thm:dirac-blown-up}
	For every $\eps, \eta > 0$, there are $c > 0$ and $n_0$ such that every graph $G$ on $n \geq n_0$ vertices with $\delta(G) \geq (1/2 + \eps)n$ contains a vertex-spanning blow-up of a cycle where each cluster has size in the range  $(1 \pm \eta) c \log n$.
\end{theorem}

This result is related to recent work of the authors~\cite{LS24a} for the analogous problem in hypergraphs and follows a similar proof strategy.
However, unlike in the hypergraph setting, where the proof techniques limit us to clusters of $\poly(\log \log n)$ vertices, we are able to achieve asymptotically optimal cluster sizes in \cref{thm:dirac-blown-up}.

{Our result is related to the Bandwidth Theorem of Böttcher, Schacht and Taraz~\cite{BST09}, which solved a conjecture of Bollobás and Komlós~\cite{KS96}.
For the case of bipartite graphs, the theorem implies that a graph $G$ as in \cref{thm:dirac-blown-up} contains every $n$-vertex bipartite graph $H$ that has `sublinear bandwidth' and bounded maximum degree $\Delta = O(1)$.
One can show that such a graph $H$ is contained in a blow-up of a path $B$ with clusters of size $o(n)$ (see \cite[\S 7]{BST09}), which explains the relation to embedding spanning graph blow-ups.
Böttcher et al.~\cite{BST09} asked whether their result can be extended to graphs $H$ where $\Delta(H)$ grows with~$n$.
Partial progress in this question was made by Böttcher, Taraz and Würfl~\cite{BTW16}, who extended the result for graphs of `bounded arrangeability' but not necessarily constant maximum degree.
\cref{thm:dirac-blown-up} answers this question without further assumptions on the maximum degree when the cluster sizes of $B$ are $o(\log n)$.
It remains an open problem to determine whether this can be extended to sublinear cluster sizes (together with a suitable maximum degree condition).}

{
Beyond this, it would be interesting to understand whether \cref{thm:dirac-blown-up} can be generalised to more complex structures such as powers of cycles.
Formally, the \emph{$r$th power of a cycle} has a cyclically ordered set of vertices such that every $r$ consecutive vertices form an $r$-clique.
It was shown by Komlós, Sárközy and Szemerédi~\cite{KSS98} that, for all $k \geq 2$, the relative minimum degree threshold forcing a Hamilton $k$th power of a cycle is $1-1/(k+1)$.
We believe that our result extends as follows.
\begin{conjecture} 
	For every $\eps, \eta > 0$ and $r\geq 1$, there are $c > 0$ and $n_0$ such that every graph $G$ on $n \geq n_0$ vertices with $\delta(G) \geq (1-1/(r+1) + \eps)n$ contains a vertex-spanning blow-up of the $r$th power of a cycle with cluster sizes in the range of $(1 \pm \eta) c \log n$.
\end{conjecture}
}

\section{Proof of the main result}

In the following, we present two intermediate results, \cref{lem:simple-blow-up-cover,lem:connecting-blow-ups}, from which we then derive \cref{thm:dirac-blown-up}.

Let us begin by introducing some terminology.
Consider a \emph{set family} $\cV$, whose elements are assumed to be pairwise disjoint by the convention of this paper.
A set $X$ is \emph{$\cV$-partite} if it has at most one vertex in each part of $\cV$.
We say that $\cV$ is \emph{$m$-balanced} if $|V| = m$ for every $V \in \cV$.
Similarly, $\cV$ is \emph{$(1\pm\eta)m$-balanced} if $(1-\eta)m \leq |V| \leq  (1 + \eta) m$ for  every $V \in \cV$.
Moreover, $\cV$ is \emph{quasi $(1\pm\eta)m$-balanced} if there is a singleton $V^\ast \in \cV$ and $\cV \sm \{V^\ast\}$ is $(1\pm \eta)m$-balanced.
{We refer to this unique singleton cluster as the \emph{exceptional cluster} of $\cV$.}
Given a graph $R$ and a set family $\cV = \{V_v\}$ indexed by $v \in V(R)$, the \emph{$\cV$-blow-up} is the blow-up obtained by replacing each $v$ with the set~$V_v$.
We denote it by $R(\cV)$, and say that $R$ is the \emph{reduced graph} of the blow-up.
Our constant hierarchies are expressed in standard $\gg$-notation.
To be precise, we write $y \gg x$ to mean
that for any $y \in (0, 1]$ there exists an $x_0 \in (0,1)$
such that for all $x_0 \geq x$ the subsequent statements
hold.  Hierarchies with more constants are defined in a
similar way and are to be read from left to right following the order in which the constants are chosen.
Moreover, we tend to ignore rounding errors if the context allows~it.

Our main result is derived from the following two lemmata.
The first one states that we can cover any Dirac-type graph with vertex-disjoint blow-ups whose reduced graphs inherit the minimum degree up to a minor deviation.

\begin{lemma}[Simple blow-up cover]\label{lem:simple-blow-up-cover}
	Let $\eps \gg 1/s \gg \eta \gg c \gg 1/n$ and $m = c \log n$.
	Let  $G$ be an $n$-vertex graph with $\delta(G) \geq (1/2 + \eps) n$.
	Then the vertices of $G$ are covered by vertex-disjoint quasi $(1\pm\eta)m$-balanced blow-ups $R_1(\cV_1),\dots,R_t(\cV_t) \subset G$ such that each reduced graph $R_i$ has $s$ vertices and satisfies  $\delta(R_i) \geq (1/2 + \eps/4) s$.
\end{lemma}

{This result is reminiscent of the Regularity Lemma~\cite{Sze76}, which produces a single quasi-random blow-up with clusters of linear size and one cluster of exceptional vertices.
In contrast, \cref{lem:simple-blow-up-cover} gives a family of disjoint complete blow-ups with clusters of logarithmic size and one `exceptional' singleton cluster.
A quantitatively weaker result for hypergraphs can be found in our previous work~\cite[Lemma 7.1]{LS24a}.}

The second lemma for the proof of \cref{thm:dirac-blown-up} allows us to find a short chain of blow-ups which connects two other clusters.

\begin{lemma}[Connecting blow-ups]\label{lem:connecting-blow-ups}
	Let $\eps, c_1 \gg c_2 \gg 1/n$,
	$m = c_1 \log n$ and $m' = c_2 \log n$.
	Let $G$ be an $n$-vertex graph with $\delta(G) \geq (1/2 + \eps) n$.
	Let $\{U,V,W\}$ be a partition of $V(G)$ with $|U|=|V|=m$.
	Then there are $m'$-sets $U' \subset U$, $V' \subset V$ and $W' \subset W$ such that $G$ contains the complete bipartite graphs $K(U',W')$ and $K(V',W')$.
\end{lemma}

The proofs of \cref{lem:simple-blow-up-cover,lem:connecting-blow-ups} {are based on a combination of subsampling with Erdős--Stone-type results and can be found in \cref{section:covers} and \cref{sec:blowups}, respectively.}

{
Let us continue with a proof overview on how to derive \cref{thm:dirac-blown-up} from here.
We begin by applying \cref{lem:simple-blow-up-cover} to cover the vertex set of $G$ with disjoint blow-ups $R_1(\cV_1),\dots,R_t(\cV_t) \subset G$ such that each reduced graph $R_i$ inherits the minimum degree condition and the clusters of each family $\cV_i$ have size about $c_1 \log n$ except for the exceptional cluster $V_i^\ast$, which is a singleton.
Next, we carefully merge $V_i^\ast$ with another cluster while retaining the other properties and, for convenience, names.
By  Dirac's theorem, each $R_i$ contains a Hamilton cycle $C_i$, which means that $R_i(\cV_i)$ is spanned by a blow-up of a cycle $C_i(\cV_i)$.
In other words, we have found a partition of $V(G)$ into blow-ups of cycles $C_1,\dots,C_t$ such that each cluster has size about $c_1 \log n$.
The idea is then to successively connect the cycles $C_i(\cV_i)$ and $C_{i+1}(\cV_{i+1})$.
More precisely, we apply \cref{lem:connecting-blow-ups} to find sets $\{U_i',V_{i+1}',W_i'\}$ of size $c_2 \log n$ each with $c_1 \gg c_2$ such that the set $U_i'$ is contained in a cluster of $C_i(\cV_i)$, the set $V_{i+1}'$ is contained in a cluster of $C_{i+1}(\cV_{i+1})$ and $G$ contains $K(U_i',W_i')$ and $K(V_{i}',W_i')$ as subgraphs.
The sets $W_i$ may (arbitrarily) intersect with the rest of the graph, but we ensure that $\bigcup W_i$ only covers a small number of vertices of each cluster, which preserves balancedness.
This results in a blow-up of a cycle that spans $G$, in which the `internal' clusters have size approximately $c_1 \log n$ and the `connecting' clusters have size $c_2 \log n$ as illustrated in \cref{fig;blow-ups}.
We conclude by subdividing the internal clusters, winding around the reduced Hamilton cycle $C_i\subset R_i$, until all clusters have size about $c_2 \log n$.
Now come the details.
}


\definecolor{ftpink}{RGB}{235, 94, 141}
\definecolor{ftteal}{RGB}{0, 154, 168}
\definecolor{ftgreen}{RGB}{157, 191, 87}
\definecolor{ftblue}{RGB}{32, 143, 206}
\definecolor{ftdarkblue}{RGB}{15, 84, 153}
\definecolor{ftred}{RGB}{206, 49, 64}
\definecolor{ftdarkred}{RGB}{127, 6, 46}

\newcommand{\fc}[1]{}
\pgfdeclarelayer{background}
\pgfdeclarelayer{main}
\pgfdeclarelayer{foreground}
\pgfsetlayers{background,main,foreground}
\begin{figure}
	\centering
	\begin{tikzpicture}[scale=0.65, every node/.style={transform shape}]
		
		\tikzstyle{bigset} = [minimum size=40,circle,line width=0.5,draw opacity=1, fill=white]
		\tikzstyle{smallset} = [minimum size=16,circle,line width=0.5,draw opacity=1, fill=white]
		\tikzstyle{vertex} = [draw=black!60, fill=black=60!,circle,minimum size = 6,inner sep=0]
		
		\tikzstyle{edge} = [fill opacity=0.3]
		\tikzstyle{smalledge} = [line width=11, draw opacity=0.3]
		\tikzstyle{bigedge} = [line width=25, draw opacity=0.3]

			\coordinate (x1) at (-2.50,-3.50) {};
			\coordinate (x2) at (-0.00,-2.50) {};
			\coordinate (x3) at (0.50,-5.00) {};
			\coordinate (x4) at (-1.50,-6.00) {};
			\coordinate (x5) at (-4.00,-5.50) {};
			
			\node[bigset, draw=ftgreen, label={90: \textcolor{ftgreen}{\huge $V_i$}}] (X1) at (x1) {};
			\node[bigset, draw=ftgreen, label={90: \textcolor{ftgreen}{\huge $U_i$}}] (X2) at (x2) {};
			\node[bigset, draw=ftgreen, label={90: \fc{$x_3$}}] at (x3) {};
			\node[bigset, draw=ftgreen, label={90: \fc{$x_4$}}] at (x4) {};
			\node[bigset, draw=ftgreen, label={90: \fc{$x_5$}}] (X5) at (x5) {};
			
			\begin{pgfonlayer}{background}
				\draw[draw=ftgreen, bigedge] (x1) -- (x2);
				\draw[draw=ftgreen, bigedge] (x2) -- (x3);
				\draw[draw=ftgreen, bigedge] (x3) -- (x4);
				\draw[draw=ftgreen, bigedge] (x5) -- (x1);
				\draw[draw=ftgreen, bigedge] (x5) -- (x4); 
			\end{pgfonlayer}

				\coordinate (y1) at (6.50,-3.00) {};
				\coordinate (y2) at (9.00,-3.50) {};
				\coordinate (y3) at (5.50,-6.00) {};
				\coordinate (y4) at (4.50,-4.00) {};
				\coordinate (y5) at (8.00,-5.50) {};
			
			\node[bigset, draw=ftgreen, label={95: \textcolor{ftgreen}{\huge $V_{i+1}$}}] (Y1) at (y1) {};
			\node[bigset, draw=ftgreen, label={85: \textcolor{ftgreen}{\huge $U_{i+1}$}}] (Y2) at (y2) {};
			\node[bigset, draw=ftgreen, label={320:  \fc{$y_3$}}] at (y3) {};
			\node[bigset, draw=ftgreen, label={250: \fc{$y_4$}}] at (y4) {};
			\node[bigset, draw=ftgreen, label={250: \fc{$y_5$}}] (Y5) at (y5) {};
			
			\begin{pgfonlayer}{background}
				\draw[draw=ftgreen, bigedge] (y1) -- (y2);
				\draw[draw=ftgreen, bigedge] (y2) -- (y5);
				\draw[draw=ftgreen, bigedge] (y3) -- (y4);
				\draw[draw=ftgreen, bigedge] (y5) -- (y3);
				\draw[draw=ftgreen, bigedge] (y1) -- (y4);
			\end{pgfonlayer}
			
			\node[label={180: \textcolor{ftgreen}{\huge{$R_i(\cV_i)$}}}] (X) at (-0.5,-8) {};
			\node[label={180: \textcolor{ftgreen}{\huge{$R_{i+1}(\cV_{i+1})$}}}] (Y) at (8.5,-8) {};
			
			\coordinate (w1) at (0.2,-2.5) {};
			\coordinate (w2) at (6.5,-2.75) {};
			\coordinate (w3) at (3,-1.75) {}; 
			
			\begin{pgfonlayer}{foreground}
				\node[smallset, draw=ftpink, label={[label distance=3mm] 280: \textcolor{ftpink}{\Large $U_{i}'$}}] at (w1) {};
				\node[smallset, draw=ftpink, label={60: \textcolor{ftpink}{\Large$V_{i+1}'$}}] at (w2) {};
				\node[smallset, draw=ftpink, label={[label distance=3mm] 270: \textcolor{ftpink}{\Large$W_{i}'$}}] at (w3) {}; 
			\end{pgfonlayer}
			
			\begin{pgfonlayer}{main}
				\draw[draw=ftpink, smalledge] (w1) -- (w3);
				\draw[draw=ftpink, smalledge] (w3) -- (w2); 
			\end{pgfonlayer}

			\coordinate (d1) at (9,-3.5) {};
			\coordinate (d2) at (12,-2.75) {};
			
			\begin{pgfonlayer}{foreground}
				\node[smallset, draw=ftpink, label={[label distance=2mm] 290: \textcolor{ftpink}{\Large $U_{i+1}'$}}] at (d1) {};
				\node[smallset, draw=ftpink, label={[label distance=3mm] 270: \textcolor{ftpink}{\Large$W_{i+1}'$}}] at (d2) {}; 
			\end{pgfonlayer}
			
			\begin{pgfonlayer}{main}
				\draw[draw=ftpink, smalledge] (d1) -- (d2);
			\end{pgfonlayer}

			\coordinate (e1) at (-5.5,-2.5) {};
			\coordinate (e2) at (-2.5,-3.50) {};
			
			\begin{pgfonlayer}{foreground}
				\node[smallset, draw=ftpink, label={[label distance=2mm] 270: \textcolor{ftpink}{\Large $W_{i-1}'$}}] at (e1) {};
				\node[smallset, draw=ftpink, label={[label distance=3mm] 290: \textcolor{ftpink}{\Large$V_{i}'$}}] at (e2) {}; 
			\end{pgfonlayer}
			
			\begin{pgfonlayer}{main}
				\draw[draw=ftpink, smalledge] (e1) -- (e2);
			\end{pgfonlayer}

		\end{tikzpicture}
		
		\caption{
		The sets of  $\cK_i=\{U_i',V_{i+1}',W_i'\}$ connect the blow-ups $R_i(\cV_i)$ and $R_{i+1}(\cV_{i+1})$.}
		\label{fig;blow-ups}
	\end{figure}
	

\begin{proof}[Proof of \cref{thm:dirac-blown-up}]
	Introduce constants $s,c_1,\eta,c_2$ with $\eps \gg 1/s \gg c_1, \eta \gg c_2 \gg 1/n$ such that $\ell = c_1/c_2$ is an integer.
	(Note that the assumption $\eps \gg \eta$ only gives a stronger outcome.)
	Set $m_1 = c_1 \log n$ and $m_2 = c_2 \log n$.
	Now consider an $n$-vertex graph $G$ with $\delta(G) \geq (1/2 + \eps) n$.
	Our goal is to find a $(1 \pm 5\eta)m_2$-blow-up of a cycle, which allows us to conclude with $5\eta, c_2$ playing the rôles of $\eta,c$.
	
	By \cref{lem:simple-blow-up-cover}, the vertices of $G$ are covered by vertex-disjoint quasi $(1\pm\eta)m_1$-balanced blow-ups $\hat R_1(\hat \cV_1),\dots,\hat R_t(\hat \cV_t) \subset G$ such that each reduced graph $\hat R_i$ has order $s+1$ and satisfies $\delta(\hat R_i) \geq (1/2 + \eps/4) (s+1)$.
	We shall find a spanning blow-up of a cycle in each of the blow-ups $\hat R_i(\hat \cV_i)$.
	
	{To this end, we first merge the exceptional cluster with another cluster as follows.}
	Fix $1 \leq i \leq t$, and let $\hat V^\ast \in \hat \cV_i$ be the (unique) singleton cluster of $\hat \cV_i$ with $\hat V^\ast = \hat V_{v^\ast}$ for some $v^\ast \in V(\hat R_i)$.
	The graph $R_i = \hat R_i - v^\ast$ has $s$ vertices, and $\delta(  R_i) \geq (1/2 + \eps/4) (s+1) - 1 > s/2$.
	By Dirac's theorem, $R_i$ contains a Hamilton cycle $C_i$, say with vertices $v_1, \dotsc, v_{s}$.
	Using again that $\delta(\hat R_i) > s/2$ together with the pigeonhole principle, we find $1 \leq j < s$ such that both $v_j$ and $v_{j+2}$ are neighbours of $v^\ast$ in $\hat R_i$ (index computations modulo $s$).
	Let $\cV_i$ be obtained from $\hat \cV_i$ by adding the vertex of $\hat V^\ast$ to the cluster of $v_{j+1}$ and then removing $\hat V^\ast$ from the family.
	Note then that $R_i(\cV_i) \subseteq G$ is a $(1 \pm 2\eta)m_1$-balanced blow-up, where $R_i$ has $s$ vertices and a Hamilton cycle $C_i$.
	Let $\cV$ be the union of the families $\cV_i$ defined in this way over all $1 \leq i \leq t$, and note that $\cV$ is $(1 \pm 2 \eta)m_1$-balanced with $ts$ clusters.

	{Now we can build our spanning blown-up cycle greedily, by connecting two cycles at a time using \Cref{lem:connecting-blow-ups}.
	In each blow-up $\cV_i$, fix two distinct clusters $U_i$ and $V_i$ whose corresponding vertices in $R_i$ are consecutive in $C_i$.
	The following claim allows us to connect each $R_i(\cV_i)$ with $R_{i+1}(\cV_{i+1})$ (index computations modulo $t$).
	\begin{claim}\label{cla:dirac-blown-up-connections}
		For $1 \leq i \leq t$, there are $m_2$-balanced families $\cK_i = \{U_i',V_{i+1}',W_i'\}$ with $U_i' \subset U_i$ and $V_{i+1}' \subset V_{i+1}$ such that the complete bipartite graphs $K(U_i', W_i')$ and $K(V_{i+1}', W_i')$ are contained in $G$.
		In addition, setting $K_i = \bigcup_{1 \leq j < i} \bigcup \cK_j$ for each $1 \leq i \leq t$, the following hold:
		\begin{enumerate}[\upshape (i)]
			\item \label{itm:disjoint} $K_i$ does not intersect any of the clusters $U_{j}$ for $i \leq j \leq t$ and $V_{j}$ for $i < j \leq t+1$ (index computations modulo $t$),
			\item \label{itm:overfull}  $K_i$ intersects any cluster in at most $2 \eta m_1$ vertices.
		\end{enumerate}
	\end{claim}
	\begin{proofclaim}
		For $1 \leq i < t$, let us assume that we have already built connecting $m_2$-sized blow-ups with families $\cK_1, \dotsc, \cK_i$ satisfying the claim.
		Since $t \leq n/(1-\eta)m_1 \leq 2n/m_1$, we have that
		\begin{align*}
			|K_i| \leq 3 t m_2 \leq 6 \frac{n}{c_1 \log n} c_2 \log n = \frac{6}{\ell} n \leq \frac{\eta \eps }{16} n \, .
		\end{align*}
		We define two sets of vertices which we want to avoid when building $\cK_i$.
		Let $Y$ contain the vertices of the clusters $U_{j}$ for $i+1 \leq j \leq t$ and $V_{j}$ for $i+1 < j \leq t+1$.
		Note that $|Y| \leq 2n/((1-\eta)s) \leq \eps n /8$, since $U_i$ and $V_i$ contain a share of at most $2/((1-\eta)s)$ vertices of $\cV_i$.
		We let $X \subset V(G)$ contain the vertices of \emph{saturated} clusters $V \in \cV$, that is, those clusters $V$ such that $|V \cap K_i| \geq \eta m_1$.
		Note that there are at most $|K_i|/(\eta m_1)$ saturated clusters.
		So
		\begin{align*}
			|X| \leq (1 + 2\eta) m_1 \frac{|K_i|}{\eta m_1} \leq 2 \frac{|K_i| }{\eta} \leq \frac{\eps}{8} n \,,
		\end{align*}
		which gives $|X \cup Y \cup K_i|  \leq \eps n / 2$.
		This implies that $G' = G - (X \cup Y \cup K_i)$ has minimum degree at least $(1/2 + \eps/2)n$.
		Apply \Cref{lem:connecting-blow-ups} with $G'$, $U_i$, $V_{i+1}$, $m_1$, $m_2$ playing the rôles of $G$, $U$, $V$, $m$, $m'$ respectively.
		This yields $m_2$-sets $U_i' \subseteq V_i$, $V_{i+1}' \subseteq V_{i+1}$, and~$W_i'$, so that the complete bipartite graphs $K(U_i', W_i')$ and $K(V_{i+1}', W_i')$ are contained in~$G$.
		We finish by setting $\cK_i = \{ U_i', V_{i+1}', W_i' \}$.
		Note that $K_{i+1} = K_i \cup \bigcup \cK_i$ satisfies \ref{itm:disjoint}.
		Moreover, $K_{i+1}$ does not intersect with saturated clusters at all.
		For any non-saturated cluster, $K_{i+1}$ intersects it in at most $2 \eta m_1$ vertices, where $\eta m_1$ vertices account for the fact that the cluster is not saturated, and $3m_2 \leq \eta m_1$ account for the vertices in $K_i$.
		This gives \ref{itm:overfull}.
	\end{proofclaim}}

	Set $\cK = \bigcup_{1 \leq i \leq t} \cK_i$, and let $\cV'$ be the set family obtained from $\cV$ after removing the vertices of $\cK$ from each cluster.
	By construction, $\cK$ is $m_2$-balanced and $\cV'$ is $(1 \pm 4 \eta)m_1$-balanced.
	Moreover, there is a cycle $C$ on $ts+t = t(s+1)$ vertices such that we can index the family $\cK \cup \cV'$ by $V(C)$, so that $C(\cV' \cup \cK)$ is a spanning blow-up of a cycle in $G$, with all clusters of size at least $m_2$.

	To finish, we subdivide the clusters of $\cV'$ into $\ell$ clusters of size as close as possible.
	Since $\cV'$ is $(1 \pm 4 \eta)m_1$-balanced, the clusters obtained from this subdivision all 	have  size $(1 \pm 5 \eta) m_1/\ell = (1 \pm 5 \eta) m_2$.
	Together with the clusters of $\cK$ we obtain a set family~$\cW$, which by construction is $(1 \pm 5 \eta) m_2$-balanced.
	Now, from $C$ we can obtain a new cycle by `winding around' the clusters of $\cV'$ (using each cycle $C_i$ $\ell$ times in each $R_i$), whose vertices are in correspondence with $\cW$, so that $\cW(C) \subseteq G$ is a spanning blow-up of a cycle.
	This gives the desired $(1 \pm 5\eta)m_2$-blow-up of a cycle.
\end{proof}

\section{Tools}

For the proofs of \cref{lem:almost-blow-up-cover,lem:connecting-blow-ups}, we require a few basic results from probability and extremal hypergraph theory.
An \emph{$s$-uniform hypergraph} (or \emph{$s$-graph} for short) $P$ consists of vertices $V(P)$ and edges $E(P)$, where every  edge $e \in E(P)$ is a set of $s$ vertices.
The \emph{minimum degree}~$\delta_1(P)$ is the largest $m$ such that every vertex of $P$ is on at least $m$ edges.

\subsection*{Degree inheritance}\label{sec:inheritance}
It is well-known that sampling an induced subgraph in a Dirac-type graph approximately preserves minimum degree conditions.
We use an $s$-graph $P$, called the \emph{property graph}, to track this phenomenon locally.

\begin{lemma}[Degree inheritance]\label{lem:inheritance}
	Let $\eps \gg 1/s \gg 1/n$.
	Let  $G$ be an $n$-vertex graph with $\delta(G) \geq (1/2 + \eps) n$.
	Let $P$ be the $s$-uniform hypergraph on the vertex set $V(G)$, where an $s$-set $S$ belongs to $P$ if and only if the induced graph $G[S]$ has minimum degree at least $(1/2 + \eps/2)s$.
	Then $\delta_1(P) \geq (1 - e^{-\sqrt{s}}) \tbinom{n-1}{s-1}$.
\end{lemma}

The proof of \cref{lem:inheritance} follows from a concentration analysis.
For the sake of completeness, the details are presented in \cref{sec:inheritance}.

\subsection*{Blow-ups}

Our approach relies on the existence of complete blow-ups with clusters of polylogarithmic size in dense hypergraphs, as guaranteed by a classic result of Erd\H{o}s~\cite{Erd64}.
We use a refinement of this result for graphs by Nikiforov~\cite{Nik08b,Nik08}, which allows us to find blow-ups with clusters of {logarithmic} size, not just polylogarithmic.
The following is an immediate corollary of \cite[Theorem 1]{Nik08b}.

\begin{lemma}[Graph blow-ups]\label{lem:blowup-niki}
	Let $G$ be an $n$-vertex graph,
	and let $F$ be an $s$-vertex graph.
	Suppose that $G$ contains at least $\eps n^s$ copies of $F$.
	Then $G$ contains a blow-up of~$F$ with all clusters of size $\lfloor \eps^{r^2} \log n \rfloor$.
\end{lemma}

\section{Simple blow-up covers} \label{section:covers}

In this section, we show \cref{lem:simple-blow-up-cover}.
{The proof relies on ideas from our work on hypergraphs~\cite[\S 7]{LS24a}, where a similar lemma is proved.
Here we obtain better bounds on the cluster sizes by applying \cref{lem:blowup-niki} instead of the related result for hypergraphs by Erd\H{o}s~\cite{Erd64}.}
We derive \cref{lem:simple-blow-up-cover} from the following two results.
The first one allows us to cover all but a linear number of vertices with blow-ups.
It first appeared with constant cluster size in the setting of perfect tilings~\cite[Corollary 10.2]{Lan23}.

\begin{lemma}[Almost blow-up cover]\label{lem:almost-blow-up-cover}
	Let $\eps \gg 1/s \gg \eta \gg c \gg 1/n$ and $m = \lfloor c \log n \rfloor$.
	Let $G$ be an $n$-vertex graph with $\delta(G) \geq (1/2 + \eps) n$.
	Then all but at most~$\eta n$ vertices of~$G$ may be covered with vertex-disjoint $m$-balanced blow-ups $\cV_1(R_1),\dots,\cV_t(R_t) \subset G$ such that each reduced graph $R_i$ has $s$ vertices and satisfies  $\delta(R_i) \geq (1/2 + \eps/2) s$.
\end{lemma}

The next result allows us to pick up the remaining vertices.

\begin{lemma}[Rooted blow-ups]\label{lem:rooted-blow-ups}
	Let $\eps \gg 1/s \gg c \gg n$, and let $n^{1/2} \leq m_1 \leq (\eps/4)n$ and $m_2 = c \log m_1$.
	Let $G$ be an $n$-vertex graph with $\delta(G) \geq (1/2 + \eps)n$.
	{Let $X \subset V(G)$ be an $m_1$-set.}
	Then $G$ contains an $m_2$-balanced blow-up $T(\cW)$ such that $T$ has order $s$ and $\delta(T) \geq (1/2 + \eps/2)s$.
	Moreover, there is $W \in \cW$ with $W \subset X$ and $W' \cap X = \es$ for all $W' \in \cW$ distinct from $W$.
\end{lemma}

{We defer proofs of \cref{lem:almost-blow-up-cover,lem:rooted-blow-ups} to \cref{sec:almost-blow-up-cover} and~\cref{sec:blowups} and continue to show \cref{lem:simple-blow-up-cover}.
To explain the approach, consider constants $ c_1 \gg c_2 \gg c_3$, and let $G$ be a graph as in the statement.
We work with clusters of sizes $m_i = c_i \log n$.
To begin, take $m_1$-balanced blow-ups $R_1(\cV_1),\dots,R_{p}(\cV_p) \subset G$ obtained from \cref{lem:almost-blow-up-cover} covering all but $o(n)$ vertices.
In particular, each $R_i$ has $s$ vertices and inherits the minimum degree condition up to a small error.
It remains to include the uncovered vertices, which we denote by $X$.
To this end, we iteratively apply \cref{lem:rooted-blow-ups} to find additional disjoint $m_2$-balanced blow-ups $T_1(\cW_1),\dots,T_{p}(\cW_q)$ each covering $m_2$ vertices of $X$ that have (in total) only a few vertices in each cluster of the families $\cV_i$.
As before, each $T_j$ has $s$ vertices and approximately inherits the minimum degree condition.
We delete the vertices of each~$\cW_j$ from each~$\cV_i$ to ensure that the families are disjoint.
Note that this leaves the families~$\cV_i$ still nearly balanced.
At the end of this process, only $O(n^{3/4})$ vertices of $Y \subset X$ remain uncovered.
To finish, we cover each $u \in Y$ with a singleton cluster of a quasi $m_3$-balanced blow-up $S_u(\cU_u)$, which is again possible by an application of \cref{lem:rooted-blow-ups}.
Here, each $S_\ell$ has $s+1$ vertices and approximately inherits the minimum degree condition.
As before, one can ensure that the families $\cU_u$ are disjoint from the families $\cV_i$ and $\cW_j$ by controlling the intersection and increasing the imbalance of $\cV_i$ and $\cW_j$ slightly.
This leaves us with a complete cover of $G$ using three types of families, whose clusters are of size roughly $m_1$, roughly $m_2$ and, lastly, of size $m_3$ with one singleton cluster.
To finish, we subdivide the larger families until all families have one singleton cluster and $s$ other clusters, which are all of size roughly $m_3$.
We remark that only the third family we obtain is quasi balanced and not balanced, and therefore adding the singleton cluster to the sets in the first two families is somewhat unnecessary.
However, doing so helps us to simplify the outcome, because we obtain a partition made only of the same type of blow-ups.
Let us now turn to the details.}

\begin{proof}[Proof of \rf{lem:simple-blow-up-cover}]
	Introduce $c_1, c_2, c_3$ with $\eta \gg c_1 \gg c_2 \gg c_3 \gg 1/n$.
	Let $m_i = c_i \log n$, for $i \in \{1, 2, 3\}$.
	By \cref{lem:almost-blow-up-cover} (with $\eta^3$ playing the rôle of $\eta$), all but at most $\eta^3 n$ vertices of $G$ may be covered with pairwise vertex-disjoint $m_1$-balanced blow-ups $R_1(\cV_1),\dots,R_{p}(\cV_p) \subset~G$, where each $R_i$ has order $s$ and $\delta(R_i) \geq (1/2 + \eps/2)s$.
	The uncovered vertices, denoted by $X$, are added onto blow-ups in two steps.
	{We first include most vertices of $X$ as follows.
	\begin{claim}\label{cla:simple-blow-up-cover-almost}
		For $q \leq \eta^3 n/m_2$, there are vertex-disjoint blow-ups $T_1(\cW_1),\dots,T_{q}(\cW_q)$ such that
		\begin{enumerate}[\upshape (i)]
			\item each $\cW_i$ is an $m_2$-balanced set family with at least one cluster contained in $X$,
			\item \label{itm:simple-blow-cover-degree} each $T_i$ is an $s$-vertex graph with $\delta(T_i) \geq (1/2 + \eps/2)s$,
			\item \label{itm:simple-blow-cover-intersection} $\bigcup_i \bigcup \cW_i$ covers at most $(\eta/4) m_1$ vertices of any cluster of $\cV_1,\dots,\cV_p$ and
			\item $\bigcup_i \bigcup \cW_i$ covers all but at most $n^{3/4}$ vertices of $X$.
		\end{enumerate}
	\end{claim}
	\begin{proofclaim}
		We construct the blow-ups one after another, covering at least $m_2$ vertices of $X$ in each step until all but at most $n^{3/4}$ vertices are covered.
		This process stops after at most $q = \eta^3 n/m_2$ steps.		
		Suppose we have already constructed blow-ups $T_1(\cW_1),\dots,T_{\ell}(\cW_\ell)$ for $1 \leq \ell < q$ satisfying part~\ref{itm:simple-blow-cover-degree} and~\ref{itm:simple-blow-cover-intersection}, which cover the vertex set $ W = \bigcup_{j=1}^{\ell} \bigcup \cW_j$ and at least $\ell m_2$ vertices of $X$.
		If $|X \setminus W| \leq n^{3/4}$, we stop the process and finish with $\ell$ playing the rôle of $q$.
		Otherwise, we shall find a blow-up $T_{\ell+1}(\cW_{\ell+1}) \subset G - W$ satisfying part~\ref{itm:simple-blow-cover-degree} and~\ref{itm:simple-blow-cover-intersection}, which in addition covers $m_2$ vertices of $X$.
		Indeed, let us call a cluster $V$ of a set family $\cV_i$ \emph{saturated} if $V$ intersects in at least $(\eta/8) m_1$ elements with the vertices of $W$.
		Note that $|W| \leq sm_2 q \leq s \eta^3 n \leq \eta^2 n \leq (\eps/4) n$.
		So there are at most $h = \eta^2 n/ (\eta m_1/8) \leq 8\eta n/m_1$ saturated clusters.
		Denote by $B \subset V(G)$ the union of the vertices of the saturated clusters.
		It follows that $|B| \leq h m_1 \leq 8 \eta n \leq (\eps/4) n$.
		Let $G' = G-W-B$, and note that $G'$ has minimum degree at least $(1/2+\eps/2)n$.
		We may thus apply \cref{lem:rooted-blow-ups} with an $\lfloor n^{3/4} \rfloor$-set $X' \subset X$ playing the rôle of $X$, and~$2 c_2$ playing the rôle of~$c$.
		Since $2 c_2 \log( n^{3/4} ) \geq m_2$, this gives a blow-up $T_{\ell+1}(\cW_{\ell+1})$ such that each $\cW_i$ is an $m_2$-balanced set family with exactly one cluster contained in $X' \subset X$.
		Clearly, $T_{\ell+1}(\cW_{\ell+1})$ is disjoint from $W = \bigcup_{j=1}^{\ell} \bigcup \cW_j$.
		Moreover,~$T_{\ell+1}$ is an $n$-vertex graph satisfying $\delta(T_i) \geq (1/2 + \eps/2)s$, which gives part~\ref{itm:simple-blow-cover-degree}.
		Let $W' = W \cup \bigcup \cW_{\ell+1}$.
		Then $W'$ intersects saturated clusters in at most $(\eta/4) m_1$ vertices and unsaturated clusters in at most $(\eta/8) m_1 + sm_2 \leq (\eta/4) m_1$ vertices, in accordance with part~\ref{itm:simple-blow-cover-intersection}.
	\end{proofclaim}}
	{Fix $T_1(\cW_1),\dots,T_{q}(\cW_q)$ as in \cref{cla:simple-blow-up-cover-almost} and denote by $Y \subset X$ the set of at most $n^{3/4}$ uncovered vertices.
	Denote by $W$ the set of vertices covered by $\cW_1,\dots, \cW_q$, and observe that $|W| \leq q s m_2 \leq \eta^2 n$.
	We finish the covering at follows.}
	{\begin{claim}\label{cla:simple-blow-up-cover-complete}
		There is a family of pairwise disjoint blow-ups $S_y(\cU_y)$ for $y \in Y$ that together cover $Y$ such that each $S_i$ is a graph on $s+1$ vertices with $\delta(S_i) \geq (1/2 + \eps/2)s$.
		Moreover, $U =\bigcup_{y\in Y} \bigcup \cU_y$ is disjoint from $W$ and has at most $(\eta/4)m_1$ vertices of any cluster of $\mathcal{V}_1, \dotsc, \mathcal{V}_p$.
	\end{claim}
	\begin{proofclaim}
		We construct the blow-ups one after another.
		Suppose we have so far covered vertices $Y' \subset Y$ with blow-ups $\{S_y(\cU_y)\}_{y \in Y'}$, where each blow-up $S_y(\cU_y)$ covers  $y$ and no other vertex from $Y$.
		Suppose furthermore that $U' = \bigcup_{y\in Y'} \bigcup \cU_y$ is disjoint from $W$ and $U'$ has at most $(\eta/4)m_1$ vertices of any cluster of $\mathcal{V}_1, \dotsc, \mathcal{V}_p$.
		We have $|U'| \leq |Y|(s+1)m_2 \leq (s+1)n^{3/4} \log n \leq \eps n / 16$.
		Say that a cluster $V$ of $\bigcup_i \mathcal{V}_i$ is \emph{saturated} if $|V \cap U'| \geq (\eta / 8) m_1$.
		Denote by $h$ the number of saturated clusters and note that $h \leq |U'| / (\eta m_1 / 8)$.
		Hence, the union $D \subset V(G)$ of the vertices of saturated clusters has size at most $h m_1 \leq (\eps / 16) n$.
		So we have
		\begin{align*}
			|W| + |U'| +|Y|+|D| \leq \eta^2 n + \eps n / 16 + n^{3/4} + \eps n / 16 \leq (\eps /4)n\,.
		\end{align*}
		Fix an uncovered vertex $y \in Y \sm Y'$.
		Let $G'$ be an auxiliary graph obtained from $G- W -U'- Y - D$ by replacing $y$ with a set of $n^{3/4}$ copies $X$ of itself.
		Note that $\delta_1(G') \geq (1/2 + \eps/2) v(G')$.
		We may thus apply \cref{lem:rooted-blow-ups} to find an $m_3$-balanced blow-up $S_y(\hat \cU_y)$, where $S$ has $s+1$ vertices and satisfies $\delta(S_y) \geq (1/2 + \eps/2)s$.
		Delete from $\hat \cU_y$ all but one copy of $y$ to obtain a quasi $m_3$-balanced family $\cU_y$.
		Clearly, $S_y(  \cU_y)$ covers exactly one vertex in $Y$ and is disjoint from the blow-ups $\{S_{y'}(\cU_{y'})\}_{y' \in Y'}$ that were selected before.
		Moreover, the clusters $\cU_y$ are disjoint from $W$ and have at most $s m_3 +1 \leq(\eta/8)m_1$ vertices in any unsaturated cluster of $\cV_1,\dots,\cV_p$.
		Moreover, $\cU_y$ is disjoint from the saturated clusters.
		It follows that $\bigcup \cU_y \cup U'$ has at most $(\eta/4)m_1$ vertices of any cluster of $\mathcal{V}_1, \dotsc, \mathcal{V}_p$.
	\end{proofclaim}}
	{Fix blow-ups $S_y(\cU_y)$ for $y \in Y$ as in \cref{cla:simple-blow-up-cover-complete}.
	Let $U = \bigcup_{y \in Y} \bigcup \cU_y$.
	We delete the vertices of each $\cW_i$ and $\cU_y$ from each $\cV_j$, keeping the names for convenience.
	Note that afterwards, $\cV_j$ is still $(1\pm\eta/2)m_1$-balanced by \cref{cla:simple-blow-up-cover-almost}\ref{itm:simple-blow-cover-intersection} and \cref{cla:simple-blow-up-cover-complete}.}
	
	To finish, we have to address the following two issues.
	Firstly, the clusters of the families $\cW_i$ and $\cV_i$ are still much larger compared to $\cU_i$.
	Secondly, the families do not have singleton clusters as required for quasi balancedness.
	To deal with this, we first split each $\cW_i$ and each $\cV_j$ into $(1\pm \eta/2)((s-1)m_3)$-balanced families.
	(This can be done greedily.)
	Then, we partition each of these $(1\pm \eta/2)((s-1)m_3)$-balanced families into~$s$ quasi $(1\pm \eta)m_3$-balanced families.
	Indeed, assuming the family is $\{X_1, \dotsc, X_s\}$, one can partition each $X_i$ into $s$ sets, one of size $1$, and the other ones of size as equal as possible, that is, of size $(1 \pm \eta)m_3$.
	Then we form $s$ new families by including in each one exactly one singleton and $s-1$ many $(1 \pm \eta)m_3$-sets.
	Hence every family is now quasi $(1\pm \eta)m_3$-balanced, and we can finish with $c=c_3$.
	{Lastly, note that each of the reduced graphs $R_i$, $T_j$ and $S_y$ still has  minimum degree at least $(1/2 + \eps/4)(s+1)$ after this process.}
\end{proof}

\section{Almost blow-up covers}\label{sec:almost-blow-up-cover}

{Here we show \cref{lem:almost-blow-up-cover}, which allows us to cover most vertices of a Dirac-type graph $G$ with blow-ups whose reduced graphs approximately inherit the minimum degree condition.
Let us outline the proof.
Consider the property $s$-graph $P$ defined in \cref{lem:inheritance}, which tracks the $s$-sets $S \subset V(G)$ that inherit the minimum degree conditions of $G$.
By degree inheritance, $P$ has large enough minimum degree to contain an almost perfect matching.
In fact, we shall prove in \cref{lem:almost-perfect-regular-tuple-tiling} that one can even cover most vertices of $P$ with pairwise disjoint quasirandom $s$-partite $s$-graphs.
Afterwards, we cover most of the vertices of each quasirandom $s$-graph separately with suitable blow-ups using \cref{lem:blowup-niki}.}

To formalise this discussion, we introduce the following notion of (weak) quasirandomness for $s$-uniform hypergraphs.
For an $s$-graph~$P$, the \emph{density} of a tuple $(V_1,\dots,V_k)$ of pairwise disjoint vertex sets is
\begin{align*}
	d_P(V_1,\dots,V_k) = \frac{e_P(V_1,\dots,V_k)}{|V_1|\dotsb|V_k|}.
\end{align*}
We say that $(V_1,\dots,V_k)$ is \emph{$(\rho,d)$-lower-regular} if for $d=d_P(V_1,\dots,V_k)$ and all choices of $X_1\subset V_1,$ $\dots,$ $X_k \subset V_k$ satisfying $|X_1| \geq \rho |V_1|,\,\dots,\,|X_k| \geq \rho |V_k|$, we have
$ d_P(X_1,\dots,X_k)  \geq d- \rho$.
Denote by $\cQ(s,m,\rho,d)$ the set of $(\rho,d')$-lower-regular $s$-tuples which are $m'$-balanced, with $m' \geq m$, and $d' \geq d$.

Our next lemma shows that one can partition most of the vertices of a hypergraph with large enough minimum degree into balanced lower-regular tuples.
This is a straightforward consequence of the Weak Hypergraph Regularity Lemma~\cite[Theorem B.2]{Lan23}.
A simpler, alternative proof of this fact appears in our related work on hypergraphs~\cite[\S 7.3]{LS24a}.
For the sake of completeness, the details are presented in \cref{sec:tiling}.

For an $s$-graph $P$ and a family $\cF$ of $s$-graphs, an \emph{$\cF$-tiling} in $P$ is a set of pairwise vertex-disjoint $k$-graphs $F_1,\dots,F_\ell \subset P$ with $F_1,\dots,F_\ell \in \cF$.

\begin{lemma}[Almost perfect tiling]\label{lem:almost-perfect-regular-tuple-tiling}
	Let $1/s \gg \eta, \rho \gg \alpha \gg 1/n$ and $m=\alpha n$.
	Then every $n$-vertex $s$-graph $P$ with $\delta_1(P) \geq \left(1-1/s^2 \right) \binom{n-1}{s-1}$ contains a $\cQ(s,m,\rho, 2 \eta)$-tiling that covers all but $\eta n$ vertices.
\end{lemma}

Given this, we can easily derive \Cref{lem:almost-blow-up-cover}, the main result of this section.

\begin{proof}[Proof of \Cref{lem:almost-blow-up-cover}]
	Introduce $\alpha$ with $\eta \gg \alpha \gg c$, and set $\rho = \eta/2$.
	We let $P$ be the auxiliary $s$-uniform hypergraph on the vertex set $V(G)$, where an $s$-set $S$ belongs to $P$ if and only if the induced graph $G[S]$ has minimum degree at least $(1/2 + \eps/2)s$.
	By \cref{lem:inheritance}, we have that $\delta_1(P) \geq (1 - e^{-\sqrt{s}}) \binom{n-1}{s-1} \geq (1 - 1/s^2)\binom{n-1}{s-1}$.

	By \cref{lem:almost-perfect-regular-tuple-tiling}, the $s$-graph $P$ contains a $\cQ(s,\alpha n,\rho,\eta)$-tiling that covers all but $\eta n$ vertices.
	Consider one of the $(\rho,\eta)$-lower-regular $s$-tuples with $\alpha n$-balanced parts $\cV=\{V_1,\dots,V_s\}$.
	Note that each $\cV$-partite edge of~$P$ corresponds to some labelled $s$-vertex graph with minimum degree at least $(1/2 + \eps/2)s$.
	Note that there are at most $2^{s^2}$ different labelled $s$-vertex graphs.
	By the pigeonhole principle and {$(\rho,\eta)$-lower-regularity},
	there is an $s$-vertex graph $R_1$ with $\delta(R_1) \geq (1/2 + \eps/2)s$ such that at least $(\eta-\rho)2^{-s^2} (\alpha n)^s$ edges of $P$ correspond to a labelled copy of $R_1$.
	By \cref{lem:blowup-niki}, we may find an {$m$-balanced blow-up $R_1(\cV_1) \subset G$} with $\cV$-partite edges, which constitutes one of the desired blow-ups.
	We repeat this procedure as long as there are at least $(\eta - \rho)(\rho \alpha n)^s$ $\cV$-partite unused edges, thus obtaining pairwise vertex-disjoint blow-ups $R_1(m),\dots,R_{\ell}(m)$.
	Since the tuple $(V_1,\dots,V_s)$ is $(\rho,\eta)$-lower-regular,
	this procedure stops when all but at most $\rho \alpha n$ vertices of $\cV$ are used in each cluster of $\cV$.
	After iterating this over all regular tuples, we have thus covered all but $\rho n + \eta n \leq 2\eta n$ vertices, which allows us to conclude with~$2 \eta$ in place of $\eta$.
\end{proof}

\section{Connecting and rooting blow-ups} \label{sec:blowups}

In this section, we show \cref{lem:connecting-blow-ups,lem:rooted-blow-ups}, starting with the former.
{The approach follows the (classic) solution by Kővári, Sós and Turán~\cite{KST54} to Zarankiewicz's problem, which is based on a double-counting argument.}

\begin{proof}[Proof of \cref{lem:connecting-blow-ups}]
	Note that every vertex in $U \cup V$ has at least $(1/2 + \eps/2)n$ neighbours in $W$.
	Let {$W_U \subseteq W$} be the set of vertices of $W$ with at least $\eps m/8$ neighbours in $U$.
	Counting edges between $U$ and $W$, we get $(1/2 + \eps/2) m n \leq |W_U|m + (n-|W_U|) \eps m/8$, which gives
	\[ |W_U| \geq \frac{1 + \eps - \eps/4}{2 - \eps/4} n \geq \left( \frac{1}{2} + \frac{\eps}{2}\right)n. \]
	Now let $W_V \subseteq W$ be the set of vertices of $W$ with at least $\eps m/8$ neighbours in $V$.
	An analogous argument gives that $|W_V| \geq \left( 1/2 + \eps/2 \right)n$.
	Hence, the set $W^\ast = W_U \cap W_V$ has size at least $ \eps n$.

	Set $t = m'$.
	We claim that there are $t$-sets $W_1 \subset U$, $W_2 \subset V$ and $W_3 \subset W$ such that $G$ contains the complete bipartite graphs $K(W_1,W_3)$ and $K(W_2,W_3)$.
	Suppose otherwise.
	Then no two $t$-sets in $U$ and $V$ have more than $t-1$ common neighbours in~$W^\ast$.
	Let $\deg(w,U)$ denote the number of edges between $w$ and $U$.
	Therefore
	\begin{align*}
		\sum_{w \in W^\ast} \binom{\deg(w,U)}{t} \binom{\deg(w,V)}{t} \leq (t-1) \binom{m}{t}^2.
	\end{align*}
	On the other hand, we have
	\begin{align*}
		\sum_{w \in W^\ast} \binom{\deg(w,U)}{t} \binom{\deg(w,V)}{t} \geq  \eps n \binom{\eps m/8}{t}^2  \geq \eps n (\eps/8)^{2t} \binom{m}{t}^2,
	\end{align*}
	where we used $t \leq \eps m / 2$ in the last inequality.
	Combining both bounds, we get $ \eps n (\eps/8)^{2t} \leq t$.
	Hence, \[ \log(\eps) + \log n - \log (8/\eps) 2t \leq \log t,\]
	which contradicts the choice of $t$.
\end{proof}

It remains to show how to obtain rooted blow-ups.
{The proof combines the ideas of \cref{lem:almost-blow-up-cover} and the solutions to Zarankiewicz's problem.}

\begin{proof}[Proof of \cref{lem:rooted-blow-ups}]
	Let $G'$ be the graph obtained from $G$ by removing the edges inside~$X$.
	Since $|X| = m_1 \leq \eps n / 4$ this barely affects the minimum degree.
	So we still have $\delta(G') \geq (1/2 + 3\eps/4)n$.
	Let $P$ be the $s$-uniform hypergraph where an $s$-set forms an edge if it has exactly one vertex in $X$ and the induced subgraph $G'[S]$ has minimum degree at least $(1/2 + \eps/2)s$.
	By \cref{lem:inheritance} and crude bounding, we have $e(P) \geq |X| n^{s-1} / 2$.
	Consider a random partition of $V(G) \sm X$ by assigning each vertex uniformly to one of $s-1$ parts.
	Together with $X$, this gives a partition of $V(G)$ into $s$ parts.
	Let $P' \subset P$ be the $s$-partite $s$-graph, which contains the crossing edges.
	By a first moment argument, it follows that with positive probability $e(P') \geq |X| n^{s-1} k!/ (2k^k)$.
	Fix such a partition for the rest of the proof.
	Recall that there are at most $2^{s^2}$ labelled graphs on $s$ vertices.
	Let $R$ be the labelled graph that appears on the edges of $P'$ most frequently, where the labels correspond to the parts.
	We pass to a subgraph $P'' \subseteq P'$ keeping only the $s$-sets $S$ where $G[S]$ is isomorphic to $R$.
	So $P''$ has at least $|X| n^{s-1} k!/ (k^k2^{s^2+2}) $ edges.
	Note that by our choice we can assume that for each $S \in P''$, $G[S]$ is isomorphic to $R$, $|S \cap X| = 1$ and also $G[S \setminus X]$ is isomorphic to the same $(s-1)$-graph $R'$, no matter the choice of $S$.

	Now let $1/s \gg \gamma_1 \gg \gamma_2 \gg c$, so that $e(P'') \geq \gamma_1 |X|n^{s-1}$.
	Let $\mathcal{R}'$ be the set of ordered $(s-1)$-tuples in $G - X$, inducing a copy of $R'$, which can be extended to at least $\gamma_1 |X|/2$ copies of $R$ by adding a vertex of $X$.
	Since $e(P'') \geq \gamma_1 |X|n^{s-1}$, we have that $|\mathcal{R}'| \geq \gamma_1 n^{s-1}/2$.
	By \cref{lem:blowup-niki}, there is a blow-up $K \subset G-X$ of $R'$ with all clusters of size $t = \gamma_2 \log n$.
	Let $R'_1, \dotsc, R'_t$ be vertex-disjoint copies of $R'$ in the blow-up $K$.
	Form an auxiliary bipartite graph $B$ with vertex classes $X$ and $W = \{ R'_1, \dotsc, R'_t \}$ where $u \sim R'_i$ if $R'_i$ can be extended to a copy of $R$ by adding $u$.
	Then $B$ has at least $\gamma_1 t|X|/2$ edges.

	To conclude, it is enough to find a complete bipartite graph in $B$ with parts of size $c \ln n$.
	The existence of such a graph follows from the usual bounds for Zarankiewicz's problem.
	To be precise, let $p = c \log n$ and suppose $B$ does not have a complete bipartite graph with two classes of size $p$.
	Then each $p$-set $Y\subset W$ has at most~$p-1$ common neighbours $N_B(Y)$ in $X$.
	Thus, double-counting and convexity reveal
	\begin{align*}
		(p-1) \binom{t}{p} & \geq \sum_{Y \in \binom{W}{p}} |N_B(Y)| = \sum_{v \in X} \binom{\deg_B(v)}{p} \\&\geq |X| \binom{e(B) /|X|}{p} \geq |X| \binom{\gamma_1 t/2}{p} \geq (\gamma_1/4)^p |X| \binom{ t}{p},
	\end{align*}
	where we used that $ p = c \log n \leq \gamma_2 \log n = \gamma_1 t / 4$ in the last inequality.
	Hence
	\[ p \geq (\gamma_1/4)^p |X|
		\geq |X| n^{-1/4} \geq n^{1/4},
	\]
	where the second inequality follows from the (small enough) choice of $c$,
	and the last inequality follows from $|X|=m_1 \geq n^{1/2}$.
	This contradicts the choice of $p$, and proves the result.
\end{proof}

\section*{Acknowledgements}
{We are grateful for the insightful comments of an anonymous referee, which helped to improve the presentation and correctness of the arguments.}

Richard Lang was supported by the Ramón y Cajal programme (RYC2022-038372-I) and by grant PID2023-147202NB-I00 funded
by MICIU/AEI/10.13039/501100011033.
Nicolás Sanhueza-Matamala was supported by ANID-FONDECYT Iniciación Nº11220269 grant and ANID-FONDECYT Regular Nº1251121 grant.

\bibliographystyle{plain}
\bibliography{../bibliography.bib}

\begin{thebibliography}{10}

\bibitem{BST09}
J.~B\"{o}ttcher, M.~Schacht, and A.~Taraz.
\newblock Proof of the bandwidth conjecture of {B}ollob{\'a}s and {K}oml{\'o}s.
\newblock {\em Math. Ann.}, 343(1):175--205, 2009.

\bibitem{BTW16}
J.~B\"ottcher, A.~Taraz, and A.~W\"urfl.
\newblock Spanning embeddings of arrangeable graphs with sublinear bandwidth.
\newblock {\em Random Structures Algorithms}, 48(2):270--289, 2016.

\bibitem{DH81}
D.~E. Daykin and R.~H{\"a}ggkvist.
\newblock Degrees giving independent edges in a hypergraph.
\newblock {\em Bull. Aust. Math. Soc.}, 23(1):103--109, 1981.

\bibitem{Dir52}
G.~A. Dirac.
\newblock Some theorems on abstract graphs.
\newblock {\em Proc. Lond. Math. Soc.}, 3(1):69--81, 1952.

\bibitem{Erd64}
P.~Erd\H{o}s.
\newblock On extremal problems of graphs and generalized graphs.
\newblock {\em Israel J. Math.}, 2:183--190, 1964.

\bibitem{ES46}
P.~Erd\H{o}s and A.~H. Stone.
\newblock On the structure of linear graphs.
\newblock {\em Bull. Amer. Math. Soc.}, 52:1087--1091, 1946.

\bibitem{KSS98}
J.~Koml\'{o}s, G.~N. S\'{a}rk\"{o}zy, and E.~Szemer{\'e}di.
\newblock On the {P}\'{o}sa-{S}eymour {c}onjecture.
\newblock {\em J. Graph Theory}, 29(3):167--176, 1998.

\bibitem{KS96}
J.~Koml{\'o}s and M.~Simonovits.
\newblock {S}zemer{\'e}di's {R}egularity {L}emma and its applications in graph
  theory.
\newblock In D.~Mikl{\'o}s, V.~T. S{\'o}s, and T.~Sz{\H o}nyi, editors, {\em
  Combinatorics, Paul Erd{\H o}s is Eighty}, volume~2, pages 295--352. Bolyai
  Society Mathematical Studies, 1996.

\bibitem{KST54}
T.~K\"ovari, V.~T. S\'os, and P.~Tur\'an.
\newblock On a problem of {K}. {Z}arankiewicz.
\newblock {\em Colloq. Math.}, 3:50--57, 1954.

\bibitem{Lan23}
R.~Lang.
\newblock Tiling dense hypergraphs.
\newblock {\em arXiv:2308.12281}, 2023.

\bibitem{LS24a}
R.~Lang and N.~Sanhueza-Matamala.
\newblock A hypergraph bandwidth theorem.
\newblock {\em arXiv:2412.14891}, 2024.

\bibitem{Nik08b}
V.~Nikiforov.
\newblock Graphs with many copies of a given subgraph.
\newblock {\em Electron. J. Combin.}, 15(1):Note 6, 6, 2008.

\bibitem{Nik08}
V.~Nikiforov.
\newblock Graphs with many {$r$}-cliques have large complete {$r$}-partite
  subgraphs.
\newblock {\em Bull. Lond. Math. Soc.}, 40(1):23--25, 2008.

\bibitem{Sze76}
E.~Szemer{\'e}di.
\newblock Regular partitions of graphs.
\newblock {\em Colloq. Internat. CNRS}, 260:399--401, 1976.

\end{thebibliography}

\appendix

\section{Inheritance}\label{sec:inheritance}

For a $k$-graph $G$ on $n$ vertices, we denote the \emph{relative degree} of a  vertex $v \in V(G)$ by $\rdeg_G(v) = \deg_G(v) / (n-1)$.
The following instance of the inheritance principle immediately implies \cref{lem:inheritance}.

\begin{lemma}[Degree inheritance] \label{lem:inheritance-degree}
	For every $\eps > 0$, there is $C >0$ such that the following holds for all $1 \leq s \leq n$.
	Suppose $G$ is an $n$-vertex graph.
	Let $P$ be the $s$-graph on $V(G)$ with an edge $S$ whenever every vertex $v  \in S$ satisfies $\rdeg_{G[S]}(v) \geq \rdeg_G(v) - \eps$.
	Then $\delta_{1}(P) \geq \left(1- \exp({-Cs}) \right)  \tbinom{n-1}{s-1}$.
\end{lemma}

We show \cref{lem:inheritance-degree} using the following standard tail bound.
Recall that a \emph{hypergeometrically} distributed random variable $X\sim \text{H}(N,K,n)$ describes the probability of $k$ successes in $n$ draws, without replacement, from a finite population of size $N$ that contains exactly $K$ objects with that feature, where each draw is either a success or a failure.
Moreover, the expected value of $X$ is $\Exp(X) = n N/K$.

\begin{lemma}[Hoeffding's inequality]\label{lem:concentration}
	For $X\sim \text{H}(N,K,n)$ and any $0 < \ell < nK/N$, we have ${\rm Pr}( |X - \Exp(X)| \leq \ell )  \leq 2 \exp\left(-2\ell^2/n\right)$.
\end{lemma}

\begin{proof}[Proof of \cref{lem:inheritance-degree}]
	Set $V = V(G)$, and fix an arbitrary vertex $x \in V$.
	Let $S \subset V$ be drawn uniformly among all $s$-sets.
	We say that a vertex {$v \in V$ is \emph{bad for $S$}} if $\rdeg_{G[S \cup \{v\}]} (v) < \rdeg_G(v) - \eps$.
	Moreover, $S$ is \emph{good} if it contains no vertex which is bad for $S$.
	So $P$ is the $s$-graph of good edges.
	It suffices to show that, conditioning on the event that $x \in S$, every vertex $v \in V$ is bad for $S$ with probability at most $ \exp(-\Omega(s))$.
	Indeed, in this case,
	\begin{align*}
		\Pr(\text{$S$ is good} \mid x \in S) & \leq \sum_{v \in V} \Pr(\text{$v$ is bad for $S$} \mid \{x,v\} \subset S)       \\
		                                     & \leq  \sum_{v \in V} \Pr(\text{$v$ is bad for $S$} \mid x \in S) \Pr( v  \in S) \\
		                                     & = \exp({-\Omega(s)}) \sum_{v \in V} \Pr( v  \in S)
		= s \exp({-\Omega(s)}) = \exp({-\Omega(s)}).
	\end{align*}

	To obtain the above bound, consider a vertex $v \in V$ with $\delta = \rdeg_G(v) - \eps$.
	We can assume that $\delta \geq 0$ as otherwise there is nothing to show.
	Let $q$ be $1$ if $xv \in E(G)$ and $0$ otherwise.
	Let $X = \deg_{G[S \cup \{v\}]}(v)$ conditioning on $x \in S$, and note that $X \sim q+\text{H}(n-1,\deg_G(v)-q,s-1)$.
	So $\Exp (X) \geq (s-1) (\deg_G(v)-q)/(n-1) \geq  (\delta + (3/4) \eps) (s-1)$.
	It follows by \cref{lem:concentration} applied with $s-1$, $(\eps/4) (s-1)$ playing the rôles of $n$, $\ell$ that
	\begin{align*}
		{\rm Pr}\left (   X <   \delta (s-1) \right) & \leq
		{\rm Pr}\left (   \Exp (X) - X \geq  (\eps/4) (s-1) \right)                                                                              \\
		                                             & \leq 2 \exp\left( -\frac{2(\eps/4)^2}{s} (s-1)^2 \right) \leq  \exp(-\Omega(s)) .\qedhere
	\end{align*}
\end{proof}

\section{Almost perfect tilings} \label{sec:tiling}

In this section, we show \cref{lem:almost-perfect-regular-tuple-tiling} following the exposition of our work on hypergraphs~\cite[\S 7]{LS24a}.

\subsection*{Regular tuples}

At the heart of the proof is the fact that every (reasonably) large dense hypergraph admits a regular tuple of linear order.

\begin{lemma}[Regular tuple]\label{lem:density-regular-tuple}
	Let $1/s, \, d,\, \rho, \nu \gg 1/m$ and $\alpha = \exp(-\rho^{-2s})$.
	Let $P$ be an $m$-balanced $s$-partite $s$-graph with $e(P) \geq d m^s$.
	Then $P$ contains an $(\rho,d)$-lower-regular $m_1$-balanced $s$-tuple, where $\alpha m \leq m_1$.
\end{lemma}
\begin{proof}
	By assumption, the vertex set of $P$ consists of an $m$-balanced tuple $(V_1,\dots,V_s)$ of density $d_P(V_1,\dots,V_s)\geq d$.
		{Let $m' =\rho m$.}
	Suppose that $(V_1,\dots,V_s)$ is not $(\rho,d)$-lower-regular.
	Then there exist $X_1\subset V_1,$ $\dotsc,$ $X_s \subset V_s$ with $|X_1|,\dotsc,|X_s| \geq m'$ and $d_P(X_1,\dotsc,X_s) < d - \rho$.
	By averaging, we can find $Y_1\subset X_1,$ $\dotsc,$ $Y_s \subset X_s$ with $|Y_1|,\dotsc,|Y_s| = m'$ and $d_P(Y_1,\dots,Y_s) < d - \rho$.
	Partition each $V_i \setminus Y_i$ into $m'$-sized sets; by adding $Y_i$, we obtain a partition $\mathcal{P}_i$ of $V_i$ into exactly $m/m' = {(\rho)}^{-1} =: t$ sets.
	There are $t^s$ many choices for a tuple $(Z_1, \dotsc, Z_s)$ of sets where $Z_i \in \mathcal{P}_i$.
	By averaging, we can find a tuple $(Z_1,\dots,Z_s)$ of parts such that $d_P(Z_1,\dots,Z_s) \geq d + \gamma$ for $\gamma = (\rho/2) (t)^{-s} = (\rho/2) {(\rho)}^{s}$.
	If $(Z_1,\dots,Z_s)$ is $(\rho,d)$-lower-regular, we are done.
	Otherwise, we iterate this process within $P[Z_1,\dots,Z_2]$.
	In each step, the density increases by $\gamma$ and the part size decreases by a factor of at most $\rho/2$.
	The process stops after at most $1/\gamma$ steps with part sizes of at least $(\rho/2)^{1/\gamma} m \geq \exp(-\rho^{-2s}) m$, as required.
\end{proof}

\subsection*{Hypergraph matchings}\label{sec:hypmatching}

We also require the following (crude) Dirac-type bound for hypergraph matchings due to Daykin and Häggkvist~\cite{DH81}, whose  short proof is included for the sake of completeness.

\begin{lemma}\label{lem:matching}
	Let $1/s, \eps \gg 1/n$ with $n$ divisible by $s$.
	Let $P$ be an $n$-vertex $s$-graph with $\delta_1(P) \geq (1 - 1/s + \eps) \tbinom{n-1}{s-1}$.
	Then $P$ has a perfect matching.
\end{lemma}

\begin{proof}
	Set $m=n/s$.
	To begin, we randomly partition the vertices of $P$ into $k$ parts of size~$m$.
	Let $P'$ be the $n$-vertex $s$-partite $s$-graph obtained by keeping only the partite edges.
	Using standard concentration bounds (such as McDiarmid's inequality), we see that $\delta_1(P') \geq (1-1/s+\eps/2) m^{s-1}$ with high probability.
	Fix such a partition.
	We shall prove that $P'$ contains a perfect matching.

	Consider a largest matching $M$ in $P'$.
	We can assume that $M$ covers all but $s$ vertices of $P'$.
	Otherwise, just add some dummy edges to $M$.
	Let $v_1,\dots,v_s$ be the vertices not covered by $M$, where $v_i$ is in cluster $V_i$ of $P'$.
	Let $\sigma_1,\dots,\sigma_s \in \cS_s$ be arbitrary permutations with $\sigma_i(i)=i$.
	Let $X = (e_1,\dots,e_{s-1}) \in M^{s-1}$.
	For $\sigma_i$ and $1\leq j \leq s$ with $j\neq i$, let $v_{\sigma_i}(j)$ be the vertex of $e_{\sigma(i)}$ in cluster $V_j$.
	Set $v_{\sigma_i}(i) = v_i$.
	We say that $X$ is \emph{good} for index $i$ if $\{v_{\sigma_i}(1),\dots,v_{\sigma_i}(s)\} \in E(P').$
	There are at most $m^{s-1}$ choices for $X$.
	Note that at most $\binom{s-1}{2}n^{s-2} < \eps/2$ of these choices contain a repeated matching edge.
	Moreover, for fixed $i$, there are at least $(1-1/s + \eps/2) m^{s-1}$ good $X \in M^{s-1}$.
	(One for each $M$-partite neighbour of $v_i$.)
	So there is some $X$ without repeated matching edge, which is good for all $1 \leq i \leq s$.
	Now we just have to choose $\pi_1,\dots,\pi_s$ such that we may find a perfect matching between the vertices of $X$ and $v_1,\dots,v_s$.
	But this is an easy exercise.
	So we have a new matching $M'$.
	If $M$ did not have dummy edges, we are done.
	Otherwise, we have found a larger matching than $M$, which is a contradiction.
\end{proof}

\subsection*{Incrementalism}

Now we are ready to show the main result of this section.
The next lemma will turn a $\cQ$-tiling into another $\cQ$-tiling with smaller parts which covers more vertices, so it suffices to iterate until most vertices are covered.
More precisely, we apply the lemma with $1/s \gg \mu \gg \alpha$ and {$\eta = \mu/16$}, so that $1 - 1/s^2 \geq 1 - 1/s + \mu$.
In each step, we turn a $\cQ$-tiling into another $\cQ$-tiling with smaller parts such that additional $(\nu\mu/16)n = \eta^2n$ vertices are covered.
So we arrive at \cref{lem:almost-perfect-regular-tuple-tiling} after at most $\eta^{-2}$ steps.

\begin{lemma}[Regular tuple tiling increment]\label{lem:larger-matching}
	Let $ 1/s,\, \mu,\, \eta, \, d,  \,\eps \gg \alpha \gg 1/n$ with $d \leq \mu/16$ and $\eps \leq d/2$.
	Set $\gamma = \exp(-\eps^{-2s})$ and $m=\alpha n$.
	Let $\cQ_1 = \cQ(s,m,\eps,{d})$ and $\cQ_2 = \cQ(s, m',\eps',d')$,
	where $m' = (\gamma \eta \mu^2/128)m$,
	$\eps'= 2\eps$,
	and $d' = d/2$.
	Let $P$ be an $s$-graph on $n$ vertices with $\delta_1(P) \geq \left(1 - 1/s + \mu \right) \binom{n-1}{s-1}$.
	Suppose that $P$ contains a $\cQ_1$-tiling $Q_1$ on $\lambda n$ vertices with $0 \leq \lambda \leq 1-\eta$.
	Then $P$ contains a $\cQ_2$-tiling on at least $(\lambda + \mu \eta /16) n$ vertices.
\end{lemma}

Before we come to the proof of \cref{lem:larger-matching}, let us derive \cref{lem:almost-perfect-regular-tuple-tiling}.
\begin{proof}[Proof of \cref{lem:almost-perfect-regular-tuple-tiling}]
	Set $\mu = 16 \eta$, and note that $1 - 1/s^2 \geq 1-1/s + \mu$.
	Set $t = (\mu \eta/16)^{-1} = \eta^{-2}$.
	For $0 \leq i \leq t$, we set $d_i = (\mu/16)/2^i$, $\eps_i = d_i/2$, $\gamma_i = \exp(-\eps_i^{-2s})$, $\alpha_0 = 1$, $\alpha_{i+1} = \gamma_{i+1} \eta \mu^2/128$, $m_i = \alpha_i n$, $\lambda_0 = 0$ and $\lambda_{i+1} = \lambda_{i} + \eta^2$.
	We begin with $\cQ_1 = \cQ(s,\alpha_0 n,\eps_0,d_0)$ and an empty $\cQ_1$-tiling $Q_1$.
	We then apply \cref{lem:larger-matching} iteratively until all but $\eta n$ vertices are covered.
	So at the beginning of step $i$, we have a $\cQ_1$-tiling $Q_1$ on $\lambda_i n$ vertices in $P$, where $\cQ_{1} = \cQ(s,m_i,\eps_i,d_i)$.
	And at the end of step $i$, we have a $\cQ_{2,i}$-tiling on at least $(\lambda_i + \mu \eta /16) n = (\lambda_i + \eta^2)n = \lambda_{i+1} n$ vertices, where $\cQ_{2,i} = \cQ(s, m_{i+1},\eps_{i+1},d_{i+1})$.
	Clearly, this process stops after at most $t$ steps.
	So we can finish with a {$\cQ(s,\alpha_t n,\eps_t,d_t)$-tiling} that covers all but $\eta n$ vertices.
\end{proof}

Let us briefly explain how we intend to prove \cref{lem:larger-matching}.
Our plan is to extract additional regular tuples from the $(1-\lambda)n \geq \eta n$ vertices not covered by $Q_1$.
We shall do this by using the assumption that the minimum degree is $\mu \binom{n-1}{s-1}$ above the threshold for a perfect matching.
To track these gains, we define $\nu = \mu/16$ and $\beta = \nu \eta /2$.
In a first step, we find a tiling $\cQ_{\text{fresh}}$ that covers $4\beta \eta n$ additional vertices outside of $Q_1$ and some vertices inside of $Q_1$.
In a second step, we `recycle' the remainder of $Q_1$ using another tiling $\cQ_{\text{rec}}$ such that $\cQ_{\text{fresh}} \cup \cQ_{\text{rec}}$ together cover all but $2\beta \eta n$ vertices in $Q_1$.
We note that the leftover comes from the fact that removing $\cQ_{\text{fresh}}$ leaves parts of $Q_1$ a bit unbalanced.
In any case, the relative difference $2\beta n = \mu \eta /16$ presents the share of additional covered vertices at the end of the argument.

Recall that the input part sizes are $m$ and $m'$.
Since the parts of $Q_1$ may be of different size, it will be convenient to use an intermediate part size.
We therefore define $m_1 = m$, $m_2 =  (\mu/8) m_1$ and $m_3 = m' = \gamma 2\beta  m_2$.
Our approach is then to refine the parts of $Q_1$, which have size at least $m_1$, into parts of uniform size $m_2$.
We remark that $m_2$ is chosen small enough to take advantage of the excess minimum degree $\mu \binom{n-1}{s-1}$.
Moreover, $m_3$ is chosen small enough to keep the number of lost vertices due to imbalancedness (discussed above) under $2\beta n$.
Now come the details.

\begin{proof}[Proof of \cref{lem:larger-matching}]
	We set $\nu = \mu/16$ and $\beta = \nu \eta /2$.
	Moreover, let $m_1 = m$, $m_2 =  (\mu/8) m_1$ and $m_3 = m' = \gamma 2\beta  m_2$.
	We begin by defining a family $\cU$ of disjoint sets in $V(P)$, as follows.
	For each $s$-tuple in $Q_1$ with parts $\{V_1, \dotsc, V_s\}$, pick a maximal number of vertex-disjoint sets of size $m_2$ inside each $V_i$, and add all of them to $\cU$.
	It follows that the sets of $\cU$ cover all but $m_2 = (\mu/8)m_1 \leq (\mu/8) |V_i|$ vertices in each $V_i$.
	Next, we extend $\cU$ by subdividing the vertices outside of $Q_1$ into a maximal number of vertex-disjoint $m_2$-sets such that $r = |\cU|$ is divisible by $s$.
	This implies that outside of $V(Q_1)$, the family $\cU$ covers all but at most $sm_2$ vertices.
	We remark that $1/s,\mu$ are much larger than $1/r$ by the choice of $\alpha$.
	Now, let $R$ be an $s$-graph with vertex set~$\cU$ and an edge $X$ if $P$ contains at least $4\nu m_2^{s}$ $X$-partite edges.
		{So $R$ plays the rôle of a reduced graph in the context of a Regularity Lemma.}
	It is therefore not surprising that $R$ inherits the minimum degree condition of~$P$.
	\begin{claim}\label{cla:reduced-degree}
		We have $\delta_1(R) \geq \left( 1 - 1/s + \mu/2 \right) \binom{r-1}{s-1}$.
	\end{claim}
	\begin{proofclaim}
		Fix a vertex $x$ in $R$.
		There are at most $r^{s-2} m_2^{s} \leq (\mu/4) m_2^{s} \binom{r-1}{s-1}$ edges of $G$, which have one vertex in the part of $x$ and at least two vertices in some part.
		Put differently, there are at least $(1-1/s^2 + 3\mu/4) m_2^s  \binom{r-1}{s-1}$ edges in $G$, which are $\cU$-partite and have a vertex in $x$.
		Every edge of $R$ incident to $x$ can host at most $m_2^{s}$ of these edges.
		Every non-edge of $R$ incident to $x$ can host at most $4\nu m_2^s$ of these edges.
		Thus $m_2^{s} \deg_R(x) + 4\nu m_2^s \binom{r-1}{s-1}  \geq (1-1/s^2 + 3 \mu / 4) m_2^s  \binom{r-1}{s-1}.$
		Solving for $\deg_R(x)$ and recalling the choice of constant yields the desired bound.
	\end{proofclaim}
	Since $r$ is divisible by $s$, the claim and \cref{lem:matching} imply that there is a perfect matching $\cM$ in~$R$.
		{We will construct a $\cQ_2$-tiling by finding tuples inside each edge $X \in \mathcal{M}$, as follows.
			Each edge $X$ of $\cM$ corresponds to an $s$-partite family $\{U_1, \dotsc, U_s\}$.
			We construct a $\cQ_2$-tiling in $X$ greedily, by finding one tuple after another.
			Note that as long as the number of used vertices in each part $U_i$ is at most $3\nu m_2$, there remain at least $\nu m_2^s \geq d m_2^s$ edges among the leftover vertices.
			(Here we used that $d \leq \mu/16 = \nu$.)
			Within the unused vertices in $X$, select subsets $U'_i \subseteq U_i$, for all $1 \leq i \leq s$, of size $\beta m_2$ each such that $d_P(U'_1, \dotsc, U'_s) \geq d$.
			(This can be done using an averaging argument.)
			We apply \cref{lem:density-regular-tuple} with $\beta m_2$ playing the rôle of $m$, to obtain an $(\eps, d)$-lower-regular $\tilde m$-balanced $s$-tuple, where $m_3 =  2 \gamma \beta  m_2 \leq \tilde m \leq \beta m_2$.
			Note that this lower-regular $s$-tuple belongs to $\cQ_2$.
			We iterate this, to greedily find a $\cQ_2$-tiling on at least $3\nu s m_2$ vertices, and at most $3\nu s m_2 + \beta m_2$ vertices, in the $X$-partite subgraph of $P$ hosted by~$\bigcup X$.}

	Denote the union of these `fresh' $\cQ_2$-tilings, over all $X \in \mathcal{M}$, by $\cQ_{\text{fresh}}$.
	Note that $\cQ_{\text{fresh}}$ covers at least $3\nu m_2$ and at most $(3\nu + \beta)m_2$ vertices in each part of $\cU$.
	Since there are at least $(1 - \lambda){(n-sm_2)}/m_2$ parts in $\mathcal{U}$ outside $V(Q_1)$, it follows that $\cQ_{\text{fresh}}$ covers at least ${3\nu (1 - \lambda)n - s} \geq 2\nu \eta n = 4\beta n$ vertices outside $V(Q_1)$.

	Next, we define a $\cQ_2$-tiling $\cQ_{\text{rec}} \subset P$ in $V(Q_1) \sm V(\cQ_{\text{fresh}})$ to `recycle' what is left of~$Q_1$.
	Let $\{V_1, \dotsc, V_s\}$ be the parts of an ${\tilde m}$-balanced $(\eps,d)$-lower-regular $s$-tuple in $Q_1$.
	Recall that, $\cQ_{\text{fresh}}$ covers between $3\nu m_2$ and $(3\nu + \beta)m_2$ vertices in each part of $\cU$.
	This implies, in particular, that the difference of leftover vertices between any two clusters of $V_1,\dots,V_s$ is at most $\beta {\tilde m}$.
	We may therefore pick subsets $U_1 \subset V_1, \dotsc, U_s \subset V_s$ of size $m^\ast = (1 - 3\nu - \beta){\tilde m}$ among the uncovered vertices.
	Note that $(U_1,\dots,U_s)$ is still $(2\eps,d')$-lower-regular for $d' \geq d-\eps \geq d/2$ and thus in $\cQ_2$.
	Indeed, for all choices of $X_1\subset U_1,$ $\dots,$ $X_s \subset U_s$ of size at least $2\eps m^\ast \geq \eps \tilde m$ each, we have
	$ d_P(X_1,\dots,X_s)  \geq d - \eps$ by $(\eps,d)$-regularity of $(V_1,\dots,V_s)$.
	Let $\cQ_{\text{rec}} \subset P$ be the $\cQ_2$-tiling obtained this way.
	It follows that $\cQ_{\text{rec}} \cup \cQ_{\text{fresh}}$ covers all but $\beta v(Q_1) = \beta \lambda n$ vertices of $V(Q_1)$.
	It follows that $\cQ_{\text{fresh}} \cup \cQ_{\text{rec}}$ covers at least $(1 - \beta)\lambda n + 4 \beta n \geq (\lambda + 4\beta  - 2\beta )n = (\lambda + \nu \eta)n$ vertices.
\end{proof}

\end{document}